\documentclass[11pt]{amsart}
\usepackage{amsmath,amsfonts,amssymb,amscd, verbatim, graphicx}

\def\classification#1{\def\@class{#1}}
\classification{\null}

\newcommand{\outout}{{\rm Out}}

\newcommand{\gk}{{\textrm GK}}
\newcommand{\Fq}{\mathbb{F}_q}

\DeclareFontFamily{OT1}{rsfs}{}
\DeclareFontShape{OT1}{rsfs}{n}{it}{<-> rsfs10}{}
\DeclareMathAlphabet{\mathscr}{OT1}{rsfs}{n}{it}
\newtheorem{prop}{Proposition}[section]
\newtheorem{thm}[prop]{Theorem}

\newtheorem{cor}[prop]{Corollary}
\newtheorem{lem}[prop]{Lemma}

\numberwithin{equation}{section}

\begin{document}

\title[Regular maps and Euler Characteristic]{Orientable regular maps with Euler characteristic divisible by few primes}
\author{Nick Gill}
\address{Nick Gill\newline
Department of Mathematics and Statistics\newline
The Open University\newline
Milton Keynes, MK7 6AA\newline
United Kingdom}
\email{n.gill@open.ac.uk}



\begin{abstract}
Let $G$ be a $(2,m,n)$-group and let $x$ be the number of distinct primes dividing $\chi$, the Euler characteristic of $G$. We prove, first, that, apart from a finite number of known exceptions, a non-abelian simple composition factor $T$ of $G$ is a finite group of Lie type with rank $n\leq x$. This result is proved using new results connecting the prime graph of $T$ to the integer $x$.

We then study the particular cases $x=1$ and $x=2$. We give a general structure statement for $(2,m,n)$-groups which have Euler characteristic a prime power, and we construct an infinite family of these objects. We also give a complete classification of those $(2,m,n)$-groups which are almost simple and for which the Euler characteristic is a prime power (there are four such).

Finally we announce a result pertaining to those $(2,m,n)$-groups which are almost simple and for which $|\chi|$ is a product of two prime powers. All such groups which are not isomorphic to $PSL_2(q)$ or $PGL_2(q)$ are completely classified.
\end{abstract}

\maketitle

\section{Introduction}

Let $m$ and $n$ be positive integers. A {\it $(2,m,n)$-group} is a triple $(G,g,h)$ where $G$ is a group, $g$ (resp. $h$) is an element of $G$ of order $m$ (resp. $n$), and $G$ has a presentation of form
\begin{equation}\label{e: meteor}
\langle g, h \, \mid \, g^m = h^n = (gh)^2 = \cdots = 1\rangle.
\end{equation}
We will often abuse notation and simply refer to the group $G$ as a $(2,m,n)$-group. Clearly a group $G$ is a $(2,m,n)$-group if and only if it is a quotient of the group
$$\Gamma(m,n)=\langle \mathfrak{g}, \mathfrak{h} \, \mid \, \mathfrak{g}^m=\mathfrak{h}^n=(\mathfrak{g}\mathfrak{h})^2=1\rangle$$
such that the images of $\mathfrak{g}$ and $\mathfrak{h}$ in $G$ have orders $m$ and $n$ respectively.

It is well known that a $(2,m,n)$-group $G$ can be associated naturally with a map on an orientable surface $\mathcal{S}$; this connection is fully explained in the beautiful paper of Jones and Singerman \cite{js}. We note first that, if $G$ is finite, then the surface $\mathcal{S}$ is compact; furthermore in this case the group $G$ has a natural regular action on the `half-edges' of the associated map, and the map is as a result called {\it regular} in the literature.

Suppose that $(G,g,h)$ is a finite $(2,m,n)$-group. Let us write $E$ (resp. $V$, $F$) for the number of edges (resp. vertices, faces) of the map. Then $G$ acts transitively on the set of edges (resp. set of vertices, set of faces) and the stabilizer in this action is cyclic of order $2$ (resp. of order $m$, of order $n$). Now we can use these facts to calculate the Euler characteristic of the surface $\mathcal{S}$:
\begin{equation}\label{e: sunny}
\chi = V-E+F = |G|\left(\frac1{m}-\frac12+\frac1{n}\right) = -|G|\frac{mn-2m-2n}{2mn}.
\end{equation}
It is well known that $\chi$ is an even integer and, moreover, that $\chi\leq 2$.

In this paper we investigate the situation where $\chi$ is divisible by few primes. We are interested in understanding the structure of the finite $(2,m,n)$-group $(G,g,h)$ in such a situation, particularly when $G$ is non-solvable. The equation (\ref{e: sunny}) implies that the quantity $\chi$ can be thought of as a property of the group, as well as the surface, since we have an expression for $\chi$ in terms of $|G|, m$ and $n$. In what follows, then, we will refer to the {\it Euler characteristic of the $(2,m,n)$-group} $(G,g,h)$ and we will not consider the associated surface $\mathcal{S}$.

\subsection{Results}

In order to state our results we need a little notation: Fix a finite group $K$; a {\it subnormal} subgroup of $K$ is a subgroup $H$ for which there exists a chain $H_1,\dots, H_k$ of subgroups of $K$ such that $H \lhd H_1 \lhd \cdots \lhd H_k \lhd K$; a simple group $J$ is {\it a composition factor} of $K$ if there exist subnormal subgroups $H_1\lhd H_2 \leq K$ such that $H_1/H_2 \cong J$. We can state our first theorem:

\begin{thm}\label{t: bounds}
 Let $(G,g,h)$ be a finite $(2,m,n)$-group with Euler characteristic $\chi$. Suppose that $\chi$ is divisible by precisely $x$ distinct primes, and suppose that $T$ is a non-abelian composition factor of $G$. Then, with finitely many exceptions, $T$ is a finite group of Lie type of rank $n$ where $n\leq x$.
\end{thm}

By {\it finitely many exceptions} we mean that there are finitely many isomorphism classes of finite simple group that are not of the given form but may still be composition factors of $G$. By {\it rank} we mean the rank of the associated simple algebraic group; this number is equal to the number of nodes on the associated Dynkin diagram. 

This theorem is stated in more detail and proved in \S\ref{s: bound} as Proposition \ref{p: bounds}. The proof uses some simple ideas connected to the structure of Sylow subgroups (these have been used before to study regular maps; see for instance \cite{bns}), as well as properties of the prime graph of a group.


In \S\ref{s: single prime}  we study the structure of $(2,m,n)$-groups $G$ with Euler characteristic $\chi$ of form $\pm 2^a$ (note that, since $\chi$ is always even, this is the only possibility when $\chi$ is a prime power). In addition to a general structure statement for $(2,m,n)$-groups of this form we give two additional results. First, we construct an infinite family of $(2,m,n)$-groups with Euler characteristic $\chi=-2^a$ for any $a\equiv 24\mod 28$; all of these have the particular property that $SL_2(8)$ is the only non-abelian composition factor. Secondly, we prove the following result concerning {\it almost simple} $(2,m,n)$-groups (i.e. $(2,m,n)$-groups $(S,g,h)$ such that $S$ is almost simple):

\begin{thm}\label{t: almost simple 2 power}
Let $(S,g,h)$ be an almost simple $(2,m,n)$-group with Euler characteristic $\chi=\pm2^a$ for some integer $a$. Then $S$ is isomorphic to one of the following:
\begin{enumerate}
\item $PSL_2(5)$ with $\{m,n\}=\{3,5\}$ and $\chi=2$;
\item $PSL_2(7)$ with $\{m,n\}=\{3,7\}$ and $\chi=-4$;
\item $PGL_2(5)$ with $\{m,n\}=\{5,6\}$ and $\chi=-16$;
\item $PGL_2(7)$ with $\{m,n\}=\{6,7\}$ and $\chi=-64$.
\end{enumerate}
What is more there is a unique $(2,m,n)$-group in each case.
\end{thm}

We clarify the meaning of the final sentence: we are asserting that in each case, if two $(2,m,n)$-groups $(S_1,g_1, h_1)$ and $(S_2, g_2, h_2)$ occur, then they are isomorphic {\it as $(2,m,n)$-groups}, i.e. there is a group isomorphism $\phi: S_1\to S_2$ such that $\phi(g_1)=g_2$ and $\phi(h_1)=h_2$. In terms of surfaces this means that there is a unique regular map in each case, up to duality (i.e. up to a swap in the order of the two generators).

In a forthcoming paper \cite{gillb} we prove a result analogous to Theorem~\ref{t: almost simple 2 power} for $(2,m,n)$-groups which are almost simple and have Euler characteristic $\chi=- 2^as^b$ for some odd prime $s$. We announce that result here:

\begin{thm}\label{t: almost simple two primes}
Let $(S,g,h)$ be an almost simple $(2,m,n)$-group with Euler characteristic $\chi=-2^as^b$ for some odd prime $s$ and integers $a,b\geq 1$. Let $T$ be the unique non-trivial normal subgroup in $S$. Then one of the following holds:
\begin{enumerate}
 \item $T=PSL_2(q)$ for some prime power $q\geq 5$ and either $S=T$ or $S=PGL_2(q)$ or else one of the possibilities listed in Table \ref{table: main1} holds.
\begin{center}
\begin{table}
 \begin{tabular}{|c|c|c|}
 \hline
Group & $\{m,n\}$ & $\chi$ \\
\hline
$S=PSL_2(9).2\cong S_6$ & $\{5,6\} $ & $-2^5\cdot 3$ \\
$S=PSL_2(9).(C_2\times C_2)$ & $\{4,10\} $ & $-2^3 \cdot 3^3$ \\
$S=PSL_2(25).2$ & $\{6,13\}$ & $-2^5\cdot 5^3$ \\
\hline
 \end{tabular}
\caption{Some $(2,m,n)$-groups for which $\chi=-2^as^b$}\label{table: main1}
\end{table}
\end{center}

\item $S=T$, $T.2$ or $T.3$, where $T$ is a finite simple group and all possibilities are listed in Table \ref{table: main}.
\begin{center}
\begin{table}
 \begin{tabular}{|c|c|c|}
 \hline
Group & $\{m,n\}$ & $\chi$ \\
\hline
$T=SL_3(3)$& $\{4,13\}$& $-|S:T|\cdot2^2\cdot 3^5$ \\
  $S=SL_3(3)$& $\{13,13\}$& $-2^3\cdot 3^5$ \\
  $S=SL_3(5)$& $\{3,31\}$& $-2^4\cdot 5^5$ \\
  $S=PSL_3(4).2$& $\{5,14\}$& $ -2^{10}\cdot 3^2$ \\
  $S=PSL_3(4).2$& $\{10,7\}$& $ -2^7\cdot 3^4$ \\
  $S=PSL_3(4).3$& $\{15,21\}$& $ -2^5\cdot 3^6$ \\
  $S=SU_3(3).2$& $\{4,7\}$& $-2^4\cdot 3^4$ \\
  $T=SU_3(3)$& $\{6,7\}$& $-|S:T|\cdot2^7\cdot 3^2$ \\
  $S=SU_3(3)$& $\{7,7\}$& $-2^4\cdot 3^4$ \\
  $S=SU_3(4).2$& $\{6,13\}$& $-2^8\cdot 5^3$ \\
 $S=PSU_3(8)$& $\{7,19\}$ & $-2^8\cdot 3^8$ \\
$S=G_2(3).2$& $\{13,14\}$& $-2^{12}\cdot3^6$ \\
 $S=Sp_6(2)$& $\{7,10\}$& $-2^9\cdot 3^6$ \\
 $S=PSU_4(3).2$& $\{5,14\}$& $-2^{11}\cdot 3^6$ \\
 $S=PSU_4(3).2$& $\{10,7\}$& $-2^8\cdot 3^8$ \\
 $S=SL_4(2).2=S_8$& $\{10,7\}$& $-2^7\cdot3^4$ \\
 $S=S_7$& $\{10,7\}$& $-2^4\cdot3^4$ \\
 $S=A_9$& $\{10,7\}$& $-2^6\cdot3^6$ \\
$T=SU_4(2)$ & $\{5,6\}$ & $ -|S:T|\cdot2^7\cdot3^3$\\
$S=SU(4,2).2$ & $\{10,4\}$ & $-2^5\cdot 3^5$ \\
$S=SU(4,2).2$ & $\{10,5\}$ & $-2^7\cdot 3^4$ \\
$S=SU(4,2).2$ & $\{10,10\}$ & $-2^6\cdot 3^5$ \\
\hline 
 \end{tabular}
\caption{Some $(2,m,n)$-groups for which $\chi=-2^as^b$}\label{table: main}
\end{table}
\end{center}
\end{enumerate}
What is more a $(2,m,n)$-group exists in each case listed in Tables \ref{table: main1} and \ref{table: main}.
\end{thm}

Some comments about Tables \ref{table: main1} and \ref{table: main} are in order. Note, first, that for those entries of Table \ref{table: main} where we specify only $T$ (rather than $S$), there are two $(2,m,n)$-groups $(S,g,h)$ in each case: one where $S=T$ and one where $S=T.2$. 

Secondly, we note that the single degree $3$ extension and the single degree $4$ extension listed in the two tables are uniquely defined: up to isomorphism there is only one almost simple group $PSL_3(4).3$ and one almost simple group $PSL_2(9).(C_2\times C_2)$; the same comment is also true for many of the degree $2$ extensions listed, but not all. However, consulting \cite{atlas} we find that, in all but one case, the requirement that $S=T.2$ is generated by two elements of orders $m$ and $n$ prescribes the group uniquely, up to isomorphism. (In particular we observe that the entry with group $PSL_2(25).2$ in Table \ref{table: main1} is distinct from $PGL_2(25)$.) 

The non-unique case is as follows: there are three distinct groups $S=PSU_4(3).2$, all of which occur as $(2,7,10)$-groups (these are all of the almost simple degree 2 extensions of $PSU_4(3)$). 

The results of this paper suggest a number of avenues of future research. We discuss some of these more fully in the companion paper \cite{gillb} however a couple of remarks are worth making here. Firstly observe that Theorems \ref{t: almost simple 2 power} and \ref{t: almost simple two primes} suggest that a stronger version of Theorem \ref{t: bounds} should hold under the extra assumption that $G$ is almost simple; in particular we conjecture that in this case the conclusion should read {\it $T$ is a finite group of Lie type of rank $n$ where $n\lneq x$}. Theorems \ref{t: almost simple 2 power} and \ref{t: almost simple two primes} confirm this conjecture for $x\leq 2$.\footnote{One could speculate further: for instance, under the suppositions of Theorem \ref{t: bounds}, is it true that if $G$ is almost simple and $T$ is a a finite group of Lie type of rank $n=x-1$, then, with finitely many exceptions, $G\leq {\rm Inndiag}(T)$? (See \cite[Definition 2.5.10]{gls3} for a definition of ${\rm Inndiag}$.)}

In a different direction we note that much of the work in this paper carries over to the study of groups associated with {\it non-orientable} regular maps. Indeed in this situation the structure of the group has more properties that we can exploit (for instance it is generated by three involutions) and we intend to address this question in a future paper.

\subsection{The Literature}

There is a substantial body of literature classifying regular maps in terms of the Euler characteristic $\chi$ of the associated surface. A complete classification (for orientable and non-orientable surfaces) has been obtained for values of $\chi$ between $-2$ and $-200$ \cite{condersmall1,condersmall2}. When $\chi$ is prime, a complete classification is given in \cite{bns}; this is the first infinite family of surfaces for which a complete classification was obtained. This breakthrough result was followed by a complete classification when $\chi=-2p$ \cite{cst}, $\chi=-3p$ \cite{cns} and $\chi=-p^2$ \cite{cps} (where, in each case, $p$ is a prime).

Our main results fit into this general scheme, however we do not give a complete classification of all regular maps in each case but study instead the structure of the associated $(2,m,n)$-group. Our focus on the case where the group is almost simple is deliberate, since these objects have a long history of study into which our work also fits. 

This history has at its heart the question of which (almost) simple groups are $(2,3,7)$-groups (otherwise known as {\it Hurwitz} groups), work on which is surveyed in \cite{conder2}; notable results in this direction include those found in \cite{jones} for Ree groups, and \cite{jonessilver} in which those $(2,4,5)$-groups which are Suzuki groups are studied. As mentioned above the question of which groups $PSL_2(q)$ and $PGL_2(q)$ are $(2,m,n)$-groups has been studied in \cite{sah}; similar questions are studied in \cite{marion,cps2}.

More generally a result of Stein \cite{stein} implies that all finite simple groups are $(2,m,n)$-groups for some $m$ and $n$ (see also \cite{msw} for a stronger statement). The same cannot be said of almost simple groups however: any almost simple group $S$ with socle $T$ such that $S/T$ is non-cyclic of odd order, e.g. $P\Gamma L(3, 7^3)$, will fail to be a $(2,m,n)$-group.

Additional motivation for considering almost simple $(2,m,n)$-groups stems from recent work of Li and {{\v{S}}ir{\'a}{\v{n}} who seek to classify regular maps acting quasiprimitively on vertices \cite{lisiran}. The main result of \cite{lisiran} reduces the general problem of classifying all such objects to the problem of understanding a number of specific families, one of which is precisely the almost simple $(2,m,n)$-groups. 

\subsection{Acknowledgments}

I thank Jozef {{\v{S}}ir{\'a}{\v{n}} for introducing me to the study of regular maps and for many very useful discussions on this subject. In addition my colleagues Robert Brignall and Ian Short of the Open University have participated in a reading group on this subject, out of which the current paper has grown. 

I also wish to especially thank John Britnell who turned out to be the perfect person to help me understand automorphisms of simple groups. I am similarly indebted to Marston Conder, who stepped in when my muddle-headed approach to computer programming was threatening to send this paper into an infinite loop.

Finally I have been a frequent visitor to the University of Bristol during the period of research for this paper, and I wish to acknowledge the generous support of the Bristol mathematics department.

\section{Background on groups}\label{s: background}

In this section we add to the notation already esablished, and we present a number of well-known results from group theory that will be useful in the sequel.

The following notation will hold for the rest of the paper: $(G,g,h)$ is always a finite $(2,m,n)$-group; $(S,g,h)$ is always a finite almost simple $(2,m,n)$-group; $T$ is always a simple group. We use $\chi$ or $\chi_G$ to denote the Euler characteristic of the group $G$. 

For groups $H, K$ we write $H.K$ to denote an extension of $H$ by $K$; i.e. $H.K$ is a group with normal subgroup $H$ such that $H.K/ H\cong K$. In the particular situation where the extension is split we write $H\rtimes K$, i.e. we have a semi-direct product. For an integer $k$ write $H^k$ to mean $\underbrace{H\times \cdots \times H}_k$. 

For an integer $n>1$ we write $C_n$ for the cyclic group of order $n$ and $D_n$ for the dihedral group of order $n$. We also sometimes write $n$ when we meet $C_n$ particularly when we are writing extensions of simple groups; so, for instance, $T.2$ is an extension of the simple group $T$ by a cyclic group of order $2$. 

Let $K$ be a group and let us consider some important normal subgroups. For primes $p_1, \dots, p_k$, write $O_{p_1, \dots, p_k}(K)$ for the largest normal subgroup of $K$ with order equal to $p_1^{a_1}\cdots p_k^{a_k}$ for some non-negative integers $a_1, \dots, a_k$; in particular $O_2(K)$ is the largest normal $2$-group in $K$. We write $Z(K)$ for the centre of $K$ and we write $F^*(K)$ for the {\it generalized Fitting subgroup} of $K$. 

The generalized Fitting subgroup was defined originally by Bender \cite{bender} and is fully explained in \cite{aschfgt}. Note that $F^*(K)=F(K)E(K)$ where $F(K)$ is the (classical) Fitting subgroup of $K$ and $E(K)$ is the product of all quasisimple subnormal subgroups of $K$ (a group $L$ is {\it quasisimple} if it is perfect and $L/Z(L)$ is simple). We will need the fact that $C_K(F^*(K)) = Z(F^*(K))$ (recall that, if $K$ is solvable, then $C_K(F(K))=Z(F(K))$; this explains why $F^*(K)$ is called the {\it generalized} Fitting subgroup). Note that $K$ is almost simple if and only if $F^*(K)$ is a finite non-abelian simple group.

Let $a$ and $b$ be positive integers. Write $(a,b)$ for the greatest common divisor of $a$ and $b$, and $[a,b]$ for the lowest common multiple of $a$ and $b$; observe that $ab=[a,b](a,b)$. For a prime $p$ write $a_p$ for the largest power of $p$ that divides $a$; write $a_{p'}$ for $a/a_p$. Write $\Phi_i(x)$ for the $i$-th cyclotomic polynomial. For a fixed positive integer $q$ and define a prime $t$ to be a {\it primitive prime divisor for $q^a-1$} if $t$ divides $\Phi_a(q)$ but $t$ does not divides $\Phi_i(q)$ for any $i=1, \dots, a-1$. For fixed $q$ we will write $r_a$ to mean a primitive prime divisor for $q^a-1$; then we can state (a version of) Zsigmondy's theorem \cite{zsig}:

\begin{thm}\label{t: zsig}
Let $q$ be a positive integer. For all $a>1$ there exists a primitive prime divisor $r_a$ unless
\begin{enumerate}
 \item $(a,q)=(6,2)$;
\item $a=2$ and $q=2^b-1$ for some positive integer $b$.
\end{enumerate}
\end{thm}

Note that $r_1$ exists whenever $q>2$; note too that, although $r_2$ does not always exist, still, for $q>3$, there are always at least two primes dividing $q^2-1$. The following result is of similar ilk to Theorem \ref{t: zsig}; it is Mih{\u{a}}ilescu's theorem \cite{mihailescu} proving the Catalan conjecture.

\begin{thm}\label{t: catalan}
Suppose that $q=p^a$ for some prime $p$ and positive integer $a$. If $q=2^a\pm1$ and $q\neq p$, then $q=9$.
\end{thm}

If $T$ is a finite simple group of Lie type, then we use notation consistent with \cite[Definition 2.2.8]{gls3} or, equivalently, with \cite[Table 5.1.A]{kl}. In particular we write $T=T_n(q)$ to mean that $T$ is of rank $n$, and $q$ is a power of a prime $p$ (in particular, for the Suzuki-Ree groups, we choose notation so that $q$ is an integer). Using this definition, certain groups are excluded because they are non-simple, namely
$$A_1(2), A_1(3), {^2A_2(2)}, {^2B_2(2)}, C_2(2), {^2F_4(2)}, G_2(2), {^2G_2(3)},$$
and $C_2(2), {^2F_4(2)}, G_2(2), {^2G_2(3)}$ are replaced by their derived subgroups. A list of all isomorphisms between different groups of Lie type is given by \cite[Theorem 2.2.10]{gls3} and \cite[Proposition 2.9.1]{kl}. These isomorphisms imply that, for certain groups, the definition of {\it rank} is ambiguous; we will take care to ensure that all results are stated so that they hold for such groups no matter what definition of rank is used.

When $T=T_n(q)$ is of Lie type we use \cite[(4.10.1)]{gls3} to write the order of $T$ as a product 
\begin{equation}\label{e: T order}
\frac1{d}q^N\prod\limits_i \Phi_i(q)^{n_i}
\end{equation}
where $d, N, n_i$ are non-negative integers and $\Phi_i(q)$ are cyclotomic polynomials as above. We will need criteria for when a Sylow $t$-subgroup of the simple group $T$ is cyclic; the following result will be useful.

\begin{lem}\label{l: T noncyclic}
Let $T=T_n(q)$ be a non-abelian simple group of Lie type with order given by (\ref{e: T order}). If an odd prime $t$ divides $\Phi_i(q)$ and $n_i>1$ then a Sylow $t$-subgroup of $T$ is non-cyclic.
\end{lem}
\begin{proof}
Suppose first that $T$ is not equal to either $A_3(q)$ or ${^2A_3(q)}$. Then \cite[Theorem 4.10.3]{gls3} implies that a Sylow $t$-subgroup of $T$ has $t$-rank greater than $1$ and the result follows. If $T={A_3}(q)$ (resp. ${^2A_3(q)}$) then the same conclusion holds unless $t=3$ and $t$ divides $q-1$ (resp. $q+1$); in this case we check the result directly and we are done.
\end{proof}

We will want to apply the Lang-Steinberg theorem at several points in this paper. Let $T=T_n(q)$ be an untwisted group of Lie type; then $T$ is the fixed set of a Frobenius endomorphism of a simple algebraic group $T_n(\overline{\Fq})$. Suppose that $\zeta$ is a non-trivial field automorphism of $T$ or, more generally, the product of a non-trivial field automorphism of $T$ with a graph automorphism of $T$. Observe that $\zeta$ can be thought of as a restriction of an endomorphism of $T_n(\overline{\Fq})$; what is more this endomorphism has the particular property that it has a finite number of fixed points. With the notation just established the Lang-Steinberg theorem implies the following result:

\begin{prop}\label{p: lang steinberg}
Any conjugacy class of $T$ which is {\it stable} under $\zeta$ (i.e. is stabilized set-wise) must intersect $X$ non-trivially, where $X$ is the centralizer in $T_n(q)$ of $\zeta$.
\end{prop}
\begin{proof}
This is well known; see for instance \cite[3.10 and 3.12]{dignemichel}. 
 \end{proof}

For $g$ an element of a group $K$ write $o(g)$ for the order of $g$; write $g^K$ to mean the conjugacy class of $g$ in $K$; write ${\mathrm Irr}(K)$ for the set of irreducible characters of $K$. The following proposition appears as an exercise in \cite[p. 45]{isaacs}.

\begin{prop}\label{p: character}
 Let $g,h,z$ be elements of a group $K$. Define the integer
$$a_{g,h,z}=\left|\{(x,y)\in g^K\times h^K \mid xy=z\}\right|.$$
 Then
$$a_{g,h,z} = \frac{|K|}{|C_K(g)|\cdot |C_K(h)|} \sum_{\chi \in {\mathrm Irr}(K)} \frac{\chi(g)\chi(h)\overline{\chi(z)}}{\chi(1)}.$$
\end{prop}

\section{Bounding the rank}\label{s: bound}

In this section we prove (a stronger version of) Theorem \ref{t: bounds}. Recall that we assume that $(G,g,h)$ is a $(2,m,n)$-group with Euler characteristic $\chi$; we start with some lemmas concerning primes dividing $\chi$. 

\begin{lem}\label{l: lcm}
Suppose that $t$ is an odd prime dividing $|G|$. If $|G|_t>|[m,n]|_t$, then $t$ divides $\chi_G$.
\end{lem}
\begin{proof}
Observe that (\ref{e: sunny}) can be rewritten as follows:
\begin{equation*}
 \begin{aligned}
\chi &= -|G|\frac{mn-2m-2n}{2mn}
&=\frac{|G|}{[m,n]}\left(\frac{mn-2m-2n}{(m,n)}\right). 
\end{aligned}
\end{equation*}
Clearly $\frac{mn-2m-2n}{(m,n)}$ is an integer and so, in particular, $\frac{|G|}{[m,n]}$ divides $\chi$. The result follows.
\end{proof}

An immediate corollary is the following result which is a particular case of Lemma 3.2 in \cite{cps}. 

\begin{lem}\label{l: sylow}
Suppose that $t$ is an odd prime divisor of $|G|$ such that a Sylow $t$-subgroup of $G$ is not cyclic. Then $|G|_t>|[m,n]|_t$ and, in particular, $t$ divides $\chi_G$.
\end{lem}

Let $N$ be a normal subgroup of $G$. Define $m_N$ (resp. $n_N)$ to be the order of $gN$ (resp. $hN$) in $G/N$.

\begin{lem}\label{l: normal}
 Let $N$ be a normal subgroup of the $(2,m,n)$-group $(G,g,h)$. If $G/N$ is not cyclic then $(G/N, gN, hN)$ is a $(2, m_N,n_N)$-group.
\end{lem}
\begin{proof}
 If $gh\in N$ then $G/N$ is cyclic. If $G/N$ is not cyclic, then $gh\not\in N$ which implies that $(gN)(hN)$ is a non-trivial involution in $G/N$; the result follows.
\end{proof}

\begin{lem}\label{l: divisible}
Let $N$ be a normal subgroup of $G$. If an odd prime $t$ satisfies $|G/N|_t>|[m_N, n_N]|_t$ then $t$ divides $\chi_G$.
\end{lem}
\begin{proof}
Since $t$ satisfies $|G/N|_t>|[m_N, n_N]|_t$ we have
$$|G|_t > |N|_t \cdot |[m_N, n_N]|_t.$$
This in turn implies that $|G|_t> |[m,n]|_t$ and Lemma~\ref{l: lcm} implies that $t$ divides $\chi_G$.
\end{proof}

\subsection{The prime graph}

In this subsection we outline a technique that we can use to exploit Lemma \ref{l: lcm}. 

Given a finite group $K$, let $\pi(K)$ be the set of all prime divisors of its order. 
Define the {\it prime graph}, or {\it Gruenberg-Kegel graph} $\gk(K)$ as follows: the vertices of $\gk(K)$ are elements of $\pi(K)$ and two vertices $p, q$ are connected by an edge if and only if $K$ contains an element of order $pq$. We will be interested in a subgraph of $\gk(K)$ obtained by restricting $\gk(K)$ to the set of primes for which the corresponding Sylow-subgroups of $K$ are cyclic; we call this subgraph $\gk_{c}(K)$; we also define $\pi_{nc}(K)\subseteq \pi(K)$ to be the set of primes for which the corresponding Sylow-subgroups of $K$ are non-cyclic.


Consider a graph $\mathcal{G}=(\mathcal{V},\mathcal{E})$; we recall some basic graph-theoretic definitions. A set of vertices of $\mathcal{G}$ is called {\it independent} if its elements are pairwise non-adjacent; write $t(\mathcal{G})$ for the maximal number of vertices in independent sets of $\mathcal{G}$; in graph theory $t(\mathcal{G})$ is usually called an {\it independence number} of the graph. A clique of the graph $\mathcal{G}$ is a set of vertices in which every vertex is adjacent to every other vertex in the set (i.e. we have a copy of the complete graph on this set of vertices as a subgraph of $\mathcal{G}$).
Write $m(\mathcal{G})$ for the minimum number of maximal cliques that contain all vertices of $\mathcal{G}$; note the easy fact that $m(\mathcal{G})\geq t(\mathcal{G})$.

The basic observation that underpins the following proposition is that the number of primes dividing $\chi$ is at least $m(\gk(G))-3$, however by being careful we can slightly strengthen this.

\begin{prop}\label{p: gk}
Let $G$ be a finite $(2,m,n)$-group of even order with corresponding Euler characteristic $\chi$. Then the number of primes dividing $(\chi, |G|)$ is at least $$\max\{0, m(\gk_c(G))-2\} + |\pi_{nc}(G)|.$$ The number of primes dividing $(\chi, |G|)$ is also at least $m(\gk(G))-2$.
\end{prop}
\begin{proof}
We start with the first assertion. Lemma \ref{l: sylow} implies that if $t\in \pi_{nc}(G)$ then $t$ divides $\chi$. Now suppose that $t\in \pi(G)\setminus \pi_{nc}(T)$. If $t=2$ then we know that $t|\chi$. If $t$ is odd, then Lemma \ref{l: lcm} implies that if $|G|_t>|[m,n]|_t$ then $t$ divides $\chi$. All of the primes dividing $m$ lie in a maximal clique of $G$; similarly for $n$. If a Sylow $2$-subgroup of $G$ is non-cyclic then we obtain immediately that there are at least $m(\gk_c(G)-2)$ distinct odd primes in $\pi(G)\setminus \pi_{nc}(T)$ that do not divide $[m,n]$ and the result follows. 

Suppose now that a Sylow $2$-subgroup of $G$ is cyclic and non-trivial. Let $\gk_{c, 2'}(G)$ be the graph obtained by removing the vertex $2$ and all edges adjacent to it. One can easily see that $m(\gk_{c,2'}(G)) \geq m(\gk_c(G))-1$. Clearly there at least $m(\gk_{c,2'}(G))-2\geq m(\gk_c(G))-3$ distinct odd primes in $\pi(G)\setminus \pi_{nc}(T)$ that do not divide $[m,n]$; since $2$ also divides $\chi(G)$ the first assertion follows.

The second assertion is similar: There are at least $m(\gk(G))-3$ distinct odd primes that do not divide $[m,n]$; we know that $2$ divides $\chi$; thus Lemma \ref{l: lcm} implies that $m(\gk(G))-2$ distinct primes divide $(\chi, |G|)$.
\end{proof}

Note that if $G$ is a finite $(2,m,n)$-group of odd order then similar bounds hold (one must replace occurences of ``-2'' with ``-3'' in the statement). In addition an easy variation of the argument used to prove Proposition \ref{p: gk} yields the following fact: if $\gk(G)$ has an independent set of size $s$ in which all primes are odd, then the number of primes dividing $\chi$ is at least $s-1$.

In fact we can do better than Proposition \ref{p: gk} by considering normal subgroups $N$ of $G$. For convenience we restrict to the situation where $N$ has non-cyclic Sylow $2$-subgroups; more general statements are possible but not needed. In particular Burnside's theorem implies that the following result can be applied whenever $N$ is non-solvable.

\begin{prop}\label{p: gk 2}
Let $(G,g,h)$ be a finite $(2,m,n)$-group of even order and Euler characteristic $\chi$. Let $N$ be a normal subgroup of $G$ with non-cyclic Sylow $2$-subgroups. Then the number of primes dividing $\frac{|G|}{[m,n]}$ is at least
$$\max\{0, m(\gk_c(N))-2\} + |\pi_{nc}(N)|.$$ The number of primes dividing $\frac{|G|}{[m,n]}$ is also at least $m(\gk(N))-2$.
\end{prop}
\begin{proof}
We proceed as before, beginning with the first assertion. If $t\in \pi_{nc}(N)$ then $t\in \pi_{nc}(G)$ and Lemma \ref{l: sylow} implies that then $t$ divides $\chi$. Now suppose that $t\in \pi(G)\setminus \pi_{nc}(T)$; by assumption $t$ is odd. Observe that $m=o(g)=o(gN)\cdot o(n_g)$ where $n_g$ is some element of $N$; similarly $n=o(h)=o(hN)\cdot o(n_h)$ where $n_h$ is some element of $N$. Since $o(gN)$ and $o(hN)$ divide $|G/N|$ and $|G|=|G/N|\cdot |N|$, we conclude that that if $|N|_t>|[o(n_g), o(n_h)]|_t$ then $t$ divides $\frac{|G|}{[m,n]}$. Now there are at least $m(\gk_c(N))-2$ distinct odd primes that do not divide $[o(n_g), o(n_h)]$ and the first assertion follows. 

The second assertion is similar.
\end{proof}

\subsection{Applications to simple groups}

We will be interested in combining Proposition \ref{p: gk 2} with results in the literature concerning the prime graph of a simple group. In particular an exhaustive study of independence sets and independence numbers in finite simple groups has been completed in \cite{vasvdov}. Their results and Proposition \ref{p: gk 2} allow us to state the following two results which yield a stronger version of Theorem \ref{t: bounds}.

\begin{prop}\label{p: bound}
 Let $T=T_n(q)$ be a non-abelian simple group of Lie type of rank $n$. Let 
$$f(T)=\max\{0, m(\gk_c(T))-2\} + |\pi_{nc}(T)|.$$ Then the following table gives lower bounds for $f(T)$.
\begin{center}
 \begin{tabular}{|c|c|}
\hline 
  $T$ & Lower bound for $f(T)$ \\
\hline 
$T$ is classical, $q\geq 4^\dagger$ & $n$ \\
${^2B_2(q)}$ & $2$ \\
${^2G_2(q)}$ & $3$ \\
${^2F_4(q)}$, $q\geq 8$ & $4$ \\
${^2F_4(2)'}$ & $3$ \\
${^3D_4(q)}$, $q\geq 4$ & $4$ \\
${^3D_4(2)}$ or ${^3D_4(3)}$ & $3$ \\
$G_2(q)$, $q\geq 4^\dagger$ & $3$ \\
$F_4(q)$, $q\geq 4^\dagger$ & $6$ \\
$E_6(q)$, $q\geq 4^\dagger$ & $8$ \\
${^2E_6(q)}$, $q\geq 4^\dagger$ & $8$ \\
$E_7(q)$, $q\geq 4^\dagger$ & $11$ \\
$E_8(q)$, $q\geq 4^\dagger$ & $15$ \\
\hline
 \end{tabular}
\end{center}
$\dagger$ In each of these cases the lower bound when $q=3$ (resp. $q=2$) is that for $q\geq 4$ reduced by $1$ (resp. $2$).
\end{prop}
\begin{proof}
In the following proof we treat families of simple groups one at a time. For a simple group $T$ we write the order of $T$ in the form (\ref{e: T order}) and then apply Theorem \ref{t: zsig} in association with Lemma \ref{l: T noncyclic} to give a lower bound on the number of distinct primes for which the Sylow subgroups of $T$ are noncyclic; we then apply the results of \cite{vasvdov} to the graph $\gk_c(T)$ to obtain the desired lower bounds. In what follows we write $t_1, t_2$ for distinct primes that divide $q^2-1$; these exist whenever $q>3$.

First let $T=A_n(q)\cong PSL_{n+1}(q)$. If $n=1$ then the $\pi_{nc}(T)=\{2\}$ and the result follows; if $n=2, q>2$ then $\pi_{nc}(T)\supseteq\{p, r_1\}$ and the result follows; if $(n,q)=(2,2)$ then the result is trivial. Now assume that $n>2$. If $q>3$, then 
$$\pi_{nc}(T) \supseteq \{p, t_1, t_2, r_3, \dots, r_{\lfloor \frac{n+1}{2} \rfloor}\}.$$
 In addition we have an independence set in $\gk(\pi_C(T))$ equal to
$$\{r_{\lfloor \frac{n+1}{2} \rfloor+1}, \dots, r_n, r_{n+1}\}.$$
 We conclude that $f(T)\geq n$ as required. If $q=3$ then we have the same two sets but must account for the fact that $t_2$ does not exist; if $q=2$ then we have the same two sets but must account for the fact that neither $t_2$ nor $r_6$ exist. 

Let $T={^2A_n(q)} \cong PSU_{n+1}(q)$ with $n>1$; the situation is very similar to that of $A_n(q)$. If $n=2$ and $q$ is odd, then $\pi_{nc}(T)\supseteq\{2,p\}$ and the result follows; if $n=2$ and $q>2$ is even, then $\pi_{nc}(T)\supseteq\{2, r_2\}$ and the result follows; if $(n,q)=(2,2)$ then $T$ is not simple, a contradiction. Now assume that $n>2$. If $q>3$ then 
\begin{equation*}
 \begin{aligned}
\pi_{nc}(T) &\supseteq \{p, t_1, t_2\} \cup \{r_{i/2} \mid 6\leq i \leq \lfloor \frac{n+1}{2} \rfloor, i \equiv 2 \mod 4 \} \\
&\cup \{r_{2i} \mid 3\leq i \leq \lfloor \frac{n+1}{2} \rfloor, i \equiv 1 \mod 2 \} \\
&\cup \{r_{i} \mid 4\leq i \leq \lfloor \frac{n+1}{2} \rfloor, i \equiv 0 \mod 4 \}.  
 \end{aligned}
\end{equation*}

 In addition (see \cite{vasvdov}) we have an independence set in $\gk(\pi_C(T))$ equal to
\begin{equation*}
 \begin{aligned}
&\{r_{i/2} \mid \lfloor \frac{n+1}{2}  \rfloor< i \leq n+1, i \equiv 2 \mod 4 \} \\
\cup &\{r_{2i} \mid \lfloor \frac{n+1}{2} \rfloor< i \leq n+1, i \equiv 1 \mod 2 \} \\
\cup &\{r_{i} \mid \lfloor \frac{n+1}{2} \rfloor< i \leq n+1, i \equiv 0 \mod 4 \}  
 \end{aligned}
\end{equation*}
and the result follows. If $q=3$ then, just as before, we have the same two sets but must account for the fact that $t_2$ does not exist; if $q=2$ then we have the same two sets but must account for the fact that neither $t_2$ nor $r_6$ exist.

Let $T=B_n(q)$ or $C_n(q)$ with $n\geq 2$.  Note that, by \cite[Theorem 2.2.10]{gls3}, we only need to consider $B_n(q)\cong P\Omega_{2n+1}(q)$ for $q$ odd. If $q>3$ then
$$\pi_{nc}(T)\supseteq \{p, t_1, t_2, r_4, r_6\dots, r_{2\lfloor \frac{n}{2}\rfloor}\}.$$
In addition we have an independence set in $\gk(\pi_C(T))$ equal to
$$\{r_{2\lfloor \frac{n}{2}\rfloor+2}, \dots, r_{2n}\}$$
and the result follows. The same considerations as before for $q=2,3$ yield the result in these cases.

Let $T=D_n(q)$ with $n\geq 4$. If $q>3$ then
$$\pi_{nc}(T)\supseteq \{p, t_1, t_2, r_4, r_6 \dots,\dots, r_{2\lfloor \frac{n}{2}\rfloor} \}.$$
In addition we have an independence set in $\gk(\pi_C(T))$ equal to
$$\{r_{2i} \mid \lfloor \frac{n+1}{2}\rfloor < i < n \} \cup \{r_{i} \mid  \lfloor \frac{n}{2}\rfloor < i \leq n, i \equiv 1 \mod 2 \}$$
and the result follows. The same considerations as before for $q=2,3$ yield the result in these cases.

Let $T={^2D_n(q)}$ with $n\geq 4$. If $q>3$ then 
$$\pi_{nc}(T)\supseteq \{p, t_1, t_2, r_4, r_6 \dots,\dots, r_{2\lfloor \frac{n-1}{2}\rfloor} \}.$$
In addition we have an independence set in $\gk(\pi_C(T))$ equal to
$$\{r_{2i} \mid \lfloor \frac{n}{2}\rfloor < i \leq n \} \cup \{r_{i} \mid  \lfloor \frac{n}{2}\rfloor < i < n, i \equiv 1 \mod 2 \}$$
and the result follows. The same considerations as before for $q=2,3$ yield the result in these cases.

Now we consider the exceptional groups. If $T={^2B_2(q)}$ then \cite{vasvdov} implies that $t(\gk(T))=4$; similarly if $T={^2G_2(q)}$ then \cite{vasvdov} implies that $t(\gk(T))=5$; the result follows in each case. 

If $T={^2F_4(q)}$ with $q\geq 8$ then $\pi_{nc}(T)\supseteq \{p, r_1\}$ and \cite{vasvdov} implies that we have an independence set in $\gk(\pi_C(T))$ of size at least $4$. If $T={^2F_4(2)'}$  then \cite{atlas} implies that $\pi_{nc}(T)=\{2,3,5\}$ and the result follows. 

If $T={^3D_4(q)}$ with $q>3$ then \cite{kleidman3d4} implies that $\pi_{nc}(T)\supseteq \{p, t_1, t_2\}$ and we have an independence set in $\gk(\pi_C(T))$ equal to $\{r_{12}, r_3, r_6\}$. If $q=3$ then we have the same two sets but must account for the fact that $t_2$ does not exist; if $q=2$ then \cite{atlas} implies that $\pi_{nc}(T)=\{2,3,7\}$ and the result follows in each case.

If $T={G_2(q)}$ with $q\geq 4$ then $\pi_{nc}(T)\supseteq \{p, t_1, t_2\}$ and \cite{vasvdov} implies that we have an independence set in $\gk(\pi_C(T))$ equal to $\{r_3, r_6\}$. If $T=G_2(3)$  then we have the same two sets but must account for the fact that $t_2$ does not exist. (Recall that $G_2(2)$ is not simple.) 

If $T={F_4(q)}$, ${E_6(q)}$, ${^2E_6}(q)$ or $E_7(q)$ $\pi_{nc}(T)\supseteq \{p, t_1, t_2, r_3, r_4, r_6\}$. Making the usual adjustments for $q=2,3$ we obtain the result for $T=F_4(q)$. Using \cite{vasvdov} the results follow in the other cases also.

If $T=E_8(q)$ then $\pi_{nc}(T)\supseteq \{p, t_1, t_2, r_3, r_4, r_6, r_8, r_{10}, r_{12}\}$; now \cite{vasvdov} implies that we have an independence set in $\gk(\pi(T))$ equal to $\{r_7, r_9, r_{14}, r_{15}, r_{18}, r_{20}, r_{24}, r_{30}\}$ and the result follows.

\end{proof}

\begin{prop}\label{p: bounds}
Let $(G,g,h)$ be a finite $(2,m,n)$-group with corresponding Euler characteristic $\chi$. Let $T$ be a non-abelian simple group which is a composition factor of $G$. Write $x$ for the number of distinct primes dividing $(\chi, |T|)$.
\begin{enumerate}
 \item If $T=A_n$, the alternating group on $n$-letters, then $n<2p_{x+1}$ where $p_{x}$ is the $x$-th smallest prime;
\item\label{two} If $T=T_n(q)$ is a finite simple group of Lie type of rank $n$, then $n\leq x$ or $n\leq x+2$ and $q\leq 3$. 
\end{enumerate}
\end{prop}

Before we prove this result, a definition: a group $J$ is a {\it chief factor} for a group $K$ if there exist subgroups $H_1, H_2$ both normal in $K$ such that $H_2/H_1\cong J$, and there is no normal subgroup $H_3$ of $K$ such that $H_1<H_3<H_2$.

\begin{proof}
Since $T$ is a composition factor of $G$ we conclude that $T^k$ is a chief factor of $K$ for some $k\geq 1$. Let $H_1, H_2$ be normal subgroups of $G$ such that $H_1\lhd H_2$ and $H_2/H_1\cong T^k$. If $|G/H_1|_t>|[o(gH_1), o(hH_1)]|_t$ then $|G|_t>[o(g), o(h)]_t$ and so $t$ divides $\chi$. Thus it is sufficient to assume that $H_1$ is trivial and prove that $|G|_t>[o(g), o(h)]_t$, i.e. $|G|_t>[m,n]_t$.

Suppose first that $T\cong A_n$. If $k=1$ then we use the fact that, for a prime $t<\frac{n}{2}$, the Sylow $t$-subgroup of $A_n$ is non-cyclic; thus $|\pi_{nc}(T)|>x$ for $n\geq 2p_{x+1}$. Now Proposition \ref{p: gk 2}, applied with $N=H_2\cong T$ gives the result. If $k>1$ then all Sylow $t$-subgroups of $H_2\cong T^k$ are non-cyclic. Thus $|\pi_{nc}(H_1)|=|\pi(T)|>x$ for $n\geq p_x$, and Proposition \ref{p: gk 2} gives the result once again.

Now suppose that $T$ is a finite simple group of Lie type of rank $n$; define $f(T)$ as in Proposition \ref{p: bound} and observe that, for $q\geq 4$, $f(T)\geq n$. Now Proposition \ref{p: gk 2}, applied with with $N=H_2\cong T$ implies that $x\geq f(T)$ and the result follows. If $k>1$ then, again, all Sylow $t$-subgroups of $H_2\cong T^k$ are non-cyclic. Now it is a triviality that $|\pi(T)|\geq f(T)$  and so once again, provided $q\geq 4$, we have $x\geq |\pi(T)|\geq n$ and the result follows.
\end{proof}

There are obvious ways to improve Proposition \ref{p: bounds}: as the proof suggests, one can give much stronger bounds whenever $G$ has a non-simple non-abelian chief factor. As well, for particular families of finite simple groups of Lie type, we may conclude stronger bounds for the value $f(T)$ than those listed in Proposition \ref{p: bound}. Finally we note that, by considering the prime graph of a simple group $T$, we lose information about, for instance, elements of order a product of prime powers (rather than products of primes); this could be rectified by studying the {\it spectrum} of the group $T$, i.e. the maximal elements in the poset of element orders of the group.

\section{Euler characteristic $\pm2^a$}\label{s: single prime}

In this section we study the structure of the $(2,m,n)$-group $G$ under the added assumption that $\chi=\pm 2^a$ for some integer $a$. The first couple of results are independent of the classification of finite simple groups.

\begin{prop}\label{p: single prime}
Let $G$ be a non-solvable finite $(2,m,n)$-group with Euler characteristic $\chi=2^a$ for some positive integer $2$. Write $\overline{G}=G/O_2(G)$. Then $\overline{G}$ has a normal subgroup isomorphic to $M\times T_1\times\cdots T_k$ where $F^*(M)$ is cyclic of odd order, $k$ is a positive integer, $T_1,\dots, T_k$ are simple groups such that, for all $i\neq j$, $(|T_i|, T_j|)$ is a power of $2$, and $\overline{G}/(M\times T_1\dots T_k)$ is isomorphic to a subgroup of ${\mathrm Out}(T_1\times\dots\times T_k)$.
\end{prop}
\begin{proof}
Recall that $F^*(\overline{G})=F(\overline{G})E(\overline{G})$ where $F(\overline{G})$ is nilpotent, $E(\overline{G})$ is a product of quasisimple groups, and both are normal in $\overline{G}$. Since $\overline{G}=G/O_2(G)$ we conclude that $F(\overline{G})$ has odd order. 

Suppose that there exists an odd prime $t$ such that the Sylow $t$-subgroups of $F(\overline{G})$ or $E(\overline{G})$ are non-cyclic. In particular the Sylow $t$-subgroups of $G$ are non-cyclic, and so Lemma \ref{l: sylow} implies that $t$ divides $\chi$ which is a contradiction. We conclude, first, that $F(\overline{G})$ is cyclic.

Suppose, next, that $E(\overline{G})$ is trivial. Then $F^*(G)=F(G)$ is cyclic. Since $C_G(F^*(G))=Z(F^*(G))$ this implies that $G/F^*(G)$ is isomorphic to a subgroup of the automorphism group of a cyclic group; in particular $G/F^*(G)$ is solvable. Since $F^*(G)$ is solvable we conclude that $G$ is solvable, a contradiction. Thus $E(\overline{G})$ is non-trivial. What is more, since we know that the Sylow $T$-subgroups of $E(\overline{G})$ are non-cyclic, we conclude that $E(\overline{G})=E_1\times E_k$ for some quasisimple groups $E_1, \dots, E_k$ such that, for all $i\neq j$, $(|E_i|, E_j|)$ is a power of $2$.

A result of Zassenhaus \cite[Theorem 15]{zassenhaus} says that no perfect group with cyclic $q$-Sylow subgroups can contain a normal subgroup of order $q$; we conclude, therefore, that, for $i=1,\dots, k$, $Z(E_i)$ is trivial or has even order. In the latter case, since $Z(E_i)$ is normal in $\overline{G}$, we have a contradiction. We conclude that, for $i=1,\dots,k$, $E_i=T_i$, a simple group.

Now let $M=C_{\overline{G}}(T_1\times \dots \times T_k)$; since $T_1\times \dots \times T_k$ is normal in $\overline{G}$, so is $V$; what is more $F^*(V)$ is cyclic of odd order. Now, since $\overline{G}/(V\times T_1\times \dots \times T_k)$ is isomorphic to a subgroup of ${\mathrm Out}(T_1\times \dots \times T_k)$, we are done.
\end{proof}

We will strengthen Proposition \ref{p: single prime} by studying the groups $T$ and $M$ in turn, and then applying Lemma \ref{l: divisible}. 

Let us recall some basic facts about cyclic $p$-groups. Suppose that $C$ is cyclic of order $p^a$ for some prime $p$ and positive integer $a$. Then $\outout C$ is cyclic of order $p^{a-1}(p-1)$. Let $D$ be the unique subgroup of $\outout C$ of order $p-1$ and consider the group $C\rtimes D$. Take $g=(g_c, g_d)\in C\rtimes D$ and suppose that $g_d$ has order $k>1$. Then, for $i\leq k$,
$$g^i=(g_c, g_d)^i = (g_c\cdot g_c^{g_d}\cdot g_c^{g_d^2}\cdots g_c^{g_d^{i-1}}, g_d^i).$$
Now for $i<k$ we have $g_d^i\neq 1$. On the other hand for $i=k$ observe that $g_c\cdot g_c^{g_d}\cdot g_c^{g_d^2}\cdots g_c^{g_d^{k-1}}$ is fixed by $g_d$ in the action of $D$ on $C$. Since $D$ acts fixed-point-freely on $C$ (i.e. no non-trivial element of $D$ fixes a non-trivial element of $C$) we conclude that the order of $g$ is $k$, the order of $g_d$. 

The facts just described will be useful as we prove the following result concerning the group $M$.

\begin{prop}\label{p: solvable}
Let $(M,g,h)$ be a solvable finite $(2,m,n)$-group with Euler characteristic $\chi$ such that $\frac{|M|}{[m,n]}=\pm2^a$. Write $\overline{M}=M/O_2(M)$. Then $\overline{M}$ has a normal subgroup $C$ such that $C$ is cyclic and $\overline{M}/C$ is isomorphic to $\{1\}$, $C_2$ or $C_2\times C_2$.
\end{prop}
\begin{proof}
Lemmas \ref{l: divisible} and \ref{l: sylow} imply that all odd Sylow subgroups of $M$ are cyclic. Thus, in particular, $F^*(M)$ is cyclic; write $F^*(M)=C_1\times C_2\times\cdots \times C_k$ where $C_i$ is cyclic of order $p_i^{a_i}$ for some odd prime $p_i$ and positive integer $a_i$. Then $M$ is isomorphic to a subgroup of $(C_1\rtimes {\mathrm Out}C_1) \times\cdots \times(C_k\rtimes {\mathrm Out}C_k)$.

Consider the projection $M_1$ of $M$ onto the first coordinate of this direct product and write $g_1, h_1$ for the images of $g$ and $h$ in $M_1$, with $m_1$ and $n_1$ their respective orders. Then $M_1$ is isomorphic to a subgroup of $C_1\rtimes {\mathrm Out}C_1$ and, by Lemma \ref{l: divisible}, $M_1$ is cyclic or is a $(2,m,n)$-group with Euler characteristic $\chi$ such that $|M_1|_{p_1}=[m_1, n_1]_{p_1}$. If $M_1$ is cyclic then we conclude that $M_1=C_1$. 

Suppose that $M_1$ is not cyclic. Now $|{\mathrm Out}C_1|=p_1^{a_1-1}(p_1-1)$ and, since all Sylow subgroups of odd order are cyclic, we conclude that $M$ is isomorphic to a subgroup of $C_1\rtimes D_1$ where $D_1$ is the subgroup of automorphisms of $C_1$ of order $p_1-1$.

As described above the only elements of $M_1$ of order divisible by $p_1$ lie in the normal subgroup isomorphic to $C_1$. Since $|M_1|_{p_1}=[m_1, n_1]_{p_1}$ we conclude that either $g_1$ or $h_1$ has order $p_1^{a_1}$ and generates a normal $p_1$-subgroup in $M_1$. Furthermore, since $g_1h_1$ has order $2$, the discussion above implies that either $g_1$ or $h_1$ has order $2$. We conclude that $M_1$ is isomorphic to $D_{2p_1^{a_1}}$, the dihedral group of order $2p_1^{a_1}$.

The same consideration can be applied to projections onto the remaining coordinates to obtain that, in every case, the image is cyclic or dihedral. We conclude that 
$$C_1\times \cdots \times C_k \lhd M \lesssim (C_1\rtimes 2)\times \dots \times (C_k \rtimes 2).$$
Furthermore the quotient of $M$ by the normal subgroup $C_1\times \cdots \times C_k$ is clearly an elementary abelian $2$-group. Now, since $M/(C_1\times \cdots \times C_k)$ is either cyclic or a $(2,m,n)$-group (and hence, in particular, generated by at most two elements) we conclude that it is of order $1,2,$ or $4$.
\end{proof}

Next we list the possible isomorphism class of $T$ in Proposition \ref{p: single prime}.

\begin{prop}\label{p: simple one prime}
Let $(S,g,h)$ be a non-abelian almost simple finite $(2,m,n)$-group with Euler characteristic $\chi$ and suppose that $\frac{|S|}{[m,n]}=\pm 2^a$ for some integer $a$. Then $S$ is isomorphic to one of the following:
\begin{enumerate}
\item $PSL_2(p)$ for $p\geq 5$ an odd prime equal to $2^a\pm 1$ for some integer $a$;
\item $PGL_2(p)$ for $p\geq 5$ an odd prime equal to $2^a\pm 1$ for some integer $a$;
\item $SL_2(2^a)$ for $a\geq 3$ an integer;
\end{enumerate}
\end{prop}
\begin{proof}
Write $T=F^*(S)$. We must first show that $T$ is of Lie type and isomorphic to $A_1(q)$ for some $q\geq 5$. 

Suppose first that $T$ is sporadic; then, using \cite{atlas}, we see that $\pi_{NC}(T) \geq 2$, a contradiction with Proposition \ref{p: gk 2}. Next suppose that $T\cong A_n$ is alternating; if $n\geq 6$ then $T$ has non-cyclic Sylow $t$-subgroups for $t=2$ and $3$, which contradicts Proposition \ref{p: gk 2}. If $n\leq 4$ then $T$ is not simple, a contradiction, and we are left with $n=5$. But in this case $T\cong PSL_2(5)$ as listed.

Thus we need only consider the situation when $T$ is of Lie type. Referring to Proposition \ref{p: bounds} we immediately conclude that $T$ has rank $\leq 1$ or else $q\leq 3$. In fact, by referring to Propositions \ref{p: gk 2} and \ref{p: bound} we can give an explicit list of all possibilities that we need to consider. In addition to $A_1(q)$ we must rule out
\begin{equation}\label{baddies2}
A_2(2), {^2A_2(2)}, G_2(2), A_2(3), A_3(2), C_2(2)', C_2(3), C_3(2), {^2A_2}(3), {^2A_3}(2).
\end{equation}
Observe first that $A_2(2)\cong A_1(7)$ as required; next observe that ${^2A_2(2)}, G_2(2)$ are not simple and can be excluded; now the remaining groups all have non-cyclic Sylow $t$-subgroups for $t=2$ and $3$, contradicting Proposition \ref{p: gk 2}. We conclude that $T\cong A_1(q)\cong PSL_2(q)$ for some prime power $q\geq 5$.

Suppose first that $q$ is odd. If $q$ is not equal to $p$ then $T$ has non-cyclic Sylow $t$-subgroups for $t=2$ and $p$, contradicting Proposition \ref{p: gk 2}. Thus we assume that $q=p$ and so $S=PSL_2(p)$ or $PGL_2(p)$. Suppose $S=PSL_2(p)$; the maximal orders of elements in $PSL_2(p)$ are $p, \frac{p-1}{2}$ and $\frac{p+1}{2}$; in order for $mn$ to be divisible by all odd primes dividing these three numbers we must have one of them equal to a power of $2$. Similarly, if $S=PGL_2(p)$, then the maximal orders of elements are $p, p-1$ and $p+1$ and the same condition holds.

Suppose next that $q=2^a$ for some integer $a\geq 3$; we must show that $S=T$. Suppose not; then $S=T\rtimes C_f$ where $f$ divides $a$ and $C_f$ is a non-trivial cyclic group generated by a field automorphism $\delta$. Maximal orders of elements in $T$ are $2, 2^a+1$ and $2^a-1$ hence $2^a-1$ must divide one of the elements in $\{m,n\}$ and $2^a+1$ must divide the other. 

For $x\in PSL_2(q)$, consider $(x, \delta)\in PSL_2(q)\rtimes \langle\delta \rangle$ and observe that
$$y=(x,\delta)^f = x\cdot x^{\delta^{f-1}}\cdots x^{\delta^2}\cdot x^{\delta}.$$
In particular observe that $y^\delta = x^\delta\cdot y \cdot x^{-\delta}$ and we obtain that $y$ lies in a conjugacy class of $PSL_2(q)$ that is stable under $\delta$. Now we apply Proposition \ref{p: lang steinberg} and conclude that $y$ is conjugate in $PSL_2(q)$ to an element of $PSL_2(q_0)$ where $q=q_0^f$. In particular the element $(x,\delta)$ has order at most $f(q_0+1)$.

Now, since $g$ and $h$ generate $S$, we know that (using semidirect product notation) one of them must be equal to $(x,\delta)$ for some $x\in PSL_2(q)$. We conclude that $q-1 \leq f(q_0+1)$ and so $q\leq 8$. Then $q=8$ and $S=T.3$ and, since $g$ and $h$ generate $S$ and multiply to give an involution, we have $o(gT)=o(hT)=3$. We conclude that the order of both $g$ and $h$ divides $3|PSL_2(2)|$; but now $\frac{|S|}{[m,n]}$ is divisible by $7$, a contradiction.
\end{proof}

Let us gather together all the results of this section so far.

\begin{cor}\label{c: full classification}
Let $(G,g,h)$ be a finite $(2,m,n)$-group with Euler characteristic $\chi=\pm2^a$ for some positive integer $2$. Write $\overline{G}=G/O_2(G)$. Then $\overline{G}$ has a subgroup $N$ such that 
\begin{enumerate}
 \item $\overline{G}/N$ is isomorphic to a subgroup of $C_2\times C_2$;
\item $N\cong C\times T$ where $C$ is cyclic and $(|C|, |T|)=1$;
\item $T$ is trivial or $T\cong SL_2(2^a)$ for some $a\geq 3$ or $T\cong PSL_2(p)$ where $p\geq 5$ is an odd prime equal to $2^a\pm 1$ for some integer $a$.
\end{enumerate}
\end{cor}
\begin{proof}
We may assume that $G$ is non-solvable since Proposition \ref{p: solvable} implies the result in the solvable case. Then Proposition \ref{p: single prime} implies that $F^*(\overline{G}) \cong C\times T_1\times \cdots T_k$ where $C$ is cyclic, $T_1, \dots, T_k$ are simple groups such that, for all $i\neq j$, $(|T_i|, T_j|)$ is a power of $2$, and $(|C|, |T|)=1$.

The fact that $(|T_i|,|T_j|)$ is a power of $2$ for all $i\neq j$ implies that $T_i\not\cong T_j$ for all $i\neq j$. This implies, in particular, that each $T_i$ is a normal subgroup in $\overline{G}$. Let $V_i$ be the centralizer of $T_i$ in $\overline{G}$; then $\overline{G}/V_i$ is an almost simple group with a normal subgroup isomorphic to $T_i$. Let $m_i$ (resp. $n_i$) be the order of the element $(gO_2(G))V_i\in \overline{G}/V_i$ (resp. $(hO_2(G))V_i\in \overline{G}/V_i$). By Lemma \ref{l: divisible} we know that $\frac{|\overline{G}/V_i|}{[m_i, n_i]|}$ is a power of $2$.  Then Proposition \ref{p: simple one prime} gives the possible isomorphism type of $T_i$. Observe that all of the simple groups listed in Proposition \ref{p: simple one prime} have order divisible by $3$. We conclude that $k=1$ and we write $F^*(\overline{G})=C\times T$ where $T$ is isomorphic to one of the simple groups listed in Proposition \ref{p: simple one prime}.

Suppose that $\overline{G}/T$ is cyclic. Since ${\mathrm Out}(T)$ has order at most $2$, we know that $V=C_{\overline{G}}(T)$ is a normal subgroup of index at most $2$ in $\overline{G}$. Thus $V$ is cyclic of odd order and we conclude that $C=V$ and the result follows.

Suppose instead that $\overline{G}/T$ is a solvable $(2,m,n)$-group. Write $m_T$ (resp. $n_T$) for the order of the element $(gO_2(G))T\in \overline{G}/T$ (resp. $(hO_2(G))T\in \overline{G}/T$). Then, by Lemma \ref{l: divisible}, we know that $\frac{|\overline{G}/T|}{[m_T, n_T]}$ is a power of $2$. Then Proposition \ref{p: solvable} implies that $\overline{G}/T$ has a normal odd order cyclic subgroup $X$ with quotient an elementary abelian $2$-group of order at most $4$. Again, since ${\mathrm Out}(T)$ has order at most $2$, we conclude that the centralizer of $T$ in $\overline{G}$ has a normal odd-order subgroup of index dividing $4$; thus, this is $C$. The quotient $\overline{G}/(C\times T)$ is an elementary abelian $2$-group of order at most $4$, and the result follows..
\end{proof}

\subsection{Almost simple groups and an infinite family}

To some extent Cor. \ref{c: full classification} reduces the question of classifying regular maps with Euler characteristic a power of $2$ to a number theoretic question. Let us consider the problem of classifying such maps with corresponding groups $G$ such that $O_2(G)=1$. We can analyse the structure of $G$ one case at a time. 

Suppose, for instance, that $G$ has a normal subgroup $C\times T$ where $T\cong SL_2(2^a)$ for some $a\geq 3$, $C$ cyclic of odd order. Write $|C|=p_1^{a_1}\cdots p_k^{a_k}$ where $p_1, \dots, p_k$ are distinct primes. If $G=C\times T$, then it is an easy matter to see that $\{m,n\}=\{|C|(2^a-1), |C|(2^a+1)\}$ and the Euler characteristic can be calculated explicitly. 

Similarly if $|G: (C\times T)|=2$ then, up to a reordering of primes and for some $i\in\{1\,\dots, n-1\}$, we have
\begin{equation*}
 \begin{aligned}
\{m,n\} \in \{ \{2p_1^{a_1}\cdots p_i^{a_i}(2^a-1), |C|(2^a+1)\}, \{2p_1^{a_1}\cdots p_i^{a_i}(2^a+1), |C|(2^a-1)\} \}.  
 \end{aligned}
\end{equation*}
Again the Euler characteristic can be calculated explicitly.

In this section we give two results of the kind just described. First of all we completely classify the almost simple $(2,m,n)$-groups with Euler characteristic a power of $2$, thereby proving Theorem \ref{t: almost simple 2 power}; this is the substance of Propositions \ref{p: almost simple one prime} and \ref{p: exist one prime}. Furthermore we study the particular case when $T\cong SL_2(8)$ and prove, in Proposition \ref{p: infinite family}, that there are an infinite number of such maps. 

\begin{prop}\label{p: almost simple one prime}
 Let $(S,g,h)$ be an almost simple finite $(2,m,n)$-group with Euler characteristic $\chi$ and suppose that $\chi=\pm2^a$ for some integer $a$. Then $S$ is isomorphic to one of the following:
\begin{enumerate}
\item $PSL_2(5)$ with $\{m,n\}=\{3,5\}$ and $\chi=2$;
\item $PSL_2(7)$ with $\{m,n\}=\{3,7\}$ and $\chi=-4$;
\item $PGL_2(5)$ with $\{m,n\}=\{5,6\}$ and $\chi=-16$;
\item $PGL_2(7)$ with $\{m,n\}=\{6,7\}$ and $\chi=-64$.
\end{enumerate}
\end{prop}
\begin{proof}
We go through the possibilities given by Proposition \ref{p: simple one prime} and calculate the Euler characteristic.

Suppose first that $S=SL_2(q)$ with $q=2^a$ for $a\geq 3$. Then we must have $\{m,n\}=\{p-1, p+1\}$ and (\ref{e: sunny}) implies that
\begin{equation*}
 \begin{aligned}
  \chi &= q(q+1)(q-1)\left(\frac{1}{q+1} + \frac{1}{q-1} - \frac12\right) \\
&= -\frac12q(q^2-4q+1).
 \end{aligned}
\end{equation*}
Now $q^2-4q+1$ is odd, therefore we require that $q^2-4q+1=1$ and thus $q=4$ which is a contradiction. (Note that $SL_2(4)$ {\bf does} give an example of a regular map, which we shall see when we examine the isomorphic group $PSL_2(5)$.)




Suppose next that $S=PSL_2(p)$ with $p=2^a\pm 1$ a prime at least $5$. Suppose first that $p=2^a+1$; then we must have $\{m,n\}=\{p, \frac{p+1}2\}$ and (\ref{e: sunny} implies that
\begin{equation}\label{plus}
 \begin{aligned}
  \chi &= \frac12 p(p+1)(p-1)\left(\frac{2}{p+1} + \frac{1}{p} - \frac12\right) \\
&= -\frac14(p-1)(p^2-5p-2).
 \end{aligned}
\end{equation}
Since $p\equiv 1 \mod 2a$ we have that $p^2-5p-2 \equiv -6 \mod 2^a$ and we conclude that $p^2-5p-2=2$ and so $p=5$ as required.

Suppose next that $p=2^a-1$; then we must have $\{m,n\}=\{p, \frac{p-1}2\}$ and (\ref{e: sunny} implies that
\begin{equation}\label{minus}
 \begin{aligned}
  \chi &= \frac12 p(p+1)(p-1)\left(\frac{2}{p-1} + \frac{1}{p} - \frac12\right) \\
&= -\frac14(p+1)(p^2-7p+2).
 \end{aligned}
\end{equation}
Since $p\equiv -1 \mod 2a$ we have that $p^2-7p+2 \equiv 10 \mod 2^a$ and we conclude that $p^2-7p+2=2$ and so $p=7$ as required.

We are left with the possibility that $S$ is not simple, but has $F^*(S)=PSL_2(p)$; then $S=PGL_2(p)$. Since $p=2^a\pm 1$ we know that $[m,n]$ must be divisible by $p\frac{p\pm1}{2}$. The only elements with order divisible by $p$ are of order $p$ and lie inside $PSL_2(p)$; thus, since we need two elements that generate $PGL_2(p)$, the remaining element must lie outside $PSL_2(p)$ and have order divisible by $\frac{p\pm1}{2}$; this implies that the element has order $p\pm 1$.

When $p=2^a+1$ we find that 
\begin{equation}\label{plus1}
 \begin{aligned}
  \chi &= p(p+1)(p-1)\left(\frac{1}{p+1} + \frac{1}{p} - \frac12\right) \\
&= -\frac12(p-1)(p^2-3p-2).
 \end{aligned}
\end{equation}
Since $p\equiv 1 \mod 2a$ we have that $p^2-3p-2 \equiv -4 \mod 2^a$ and we conclude that $2^a=4$ and so $p=5$ as required.

When $p=2^a-1$ we find that 
\begin{equation}\label{minus1}
 \begin{aligned}
  \chi &= p(p+1)(p-1)\left(\frac{1}{p-1} + \frac{1}{p} - \frac12\right) \\
&= -\frac12(p+1)(p^2-5p+2).
 \end{aligned}
\end{equation}
Since $p\equiv -1 \mod 2a$ we have that $p^2-7p+2 \equiv 8 \mod 2^a$ and we conclude that $2^a=8$ and so $p=7$ as required.
\end{proof}

We need to check that the four examples listed in Proposition \ref{p: almost simple one prime} really do occur; in the next proposition we do this and, moreover, we show that, in each case, the listed $(2,m,n)$-group is unique.

\begin{prop}\label{p: exist one prime}
Let $S$ be one of the four groups listed in Proposition \ref{p: almost simple one prime}, and let $m$ and $n$ be the listed integers in increasing order. Then there exist $g,h\in S$ such that $(S,g,h)$ is a $(2,m,n)$-group. Furthermore if $(S,g', h')$ is a $(2,m,n)$-group then there exists a group automorphism $\phi:S\to S$ such that $\phi(g)=g'$ and $\phi(h)=h'$. 
\end{prop}
\begin{proof}
Suppose first that $S=PSL_2(q)$ with $q=5$ or $7$. We consult \cite{atlas} to obtain the maximal subgroups of $S$ and observe first that any pair of elements $g,h\in S$ of orders $m=3$ (resp. $n=q$) must generate $S$. Next we observe that, for any triple of elements $(g,h,z)$ of orders $(3,q,2)$, all non-trivial characters take a zero value at $g$, $h$ or $z$. This fact along with Proposition \ref{p: character} implies that
\begin{equation}\label{e: comet}
a_{g,h,z}=|\{(x,y)\in g^S\times h^S \, \mid \, xy=z\}=\frac{|S|}{|C_S(g)|\cdot |C_S(h)|}=\frac{|S|}{3\cdot q}=\frac{q^2-1}{6}.
\end{equation}
Thus $a_{g,h,z}>0$ and we can choose $g$ and $h$ so that the order of $g$ (resp. $h$, $gh$) is $m$ (resp. $n$, $2$), as required. Assume from here on that $g$ and $h$ have these properties.

Suppose next that $g'\in g^S$ and $h'\in h^S$ satisfy $(gh)^2=1$. Since there is only one class of involutions in $S$ we can conjugate $g'$ and $h'$ to ensure that $g'h'=gh$. Next observe that, for both $q=5$ and $q=7$, the right hand side of (\ref{e: comet}) has size equal to $|C_S(gh)|$. Now $C_S(gh)\cap C_S(g)$ is trivial and so pairs $(g^x, h^x)$ are all distinct for $x\in C_S(gh)$. Since $g^x h^x=gh$ in every case, these are all possible pairs in $(g^S, h^S)$ which multiply to give $z$; we conclude that there exists $x\in S$ such that $g'=g^x$ and $h'=h^x$.

Finally observe that, in both cases, $S$ contains a unique conjugacy class of order $3$ and two of order $q$. We must show that if $(S,g,h)$ and $(S, g, h')$ are $(2,m,n)$-groups with $h\not\in h^S$ then there exists a group automorphism $\phi:S\to S$ such that $\phi(h)=h'$.  The two conjugacy classes of order $q$ are fused in $S.2=PSL_2(q)$. Thus we set $f$ to be an element of order $6$ in $S.2$ such that $f^2=g$. Define $\phi_0:S\to S, s\mapsto s^f$; then, by assumption, $h^f\in (h')^S$. If $h^f=h'$ then take $\phi=\phi_0$ and we are done. Otherwise choose $f_1$ such that $(h^{ff_1})=h'$ and $g^{f_1}=g$ (this is possible by the previous paragraph) and set $\phi(s)=s^{ff_1}$ and we are done.

Now suppose that $S=PGL_2(q)$ with $q=5$ or $7$. We consult \cite{atlas} to obtain the maximal subgroups of $S$ and observe first that any pair of elements $g,h\in S$ of orders $m$ (resp. $n$) must generate $S$. We observe, furthermore, that exactly one of the elements $g,h$ lies in $T=PSL_2(q)$; this implies, in particular, that $gh\not\in PSL_2(q)$. 

There is a unique conjugacy class of involutions in $S\backslash T$ to which $gh$ must belong; similarly there is a unique conjugacy class of order $m$ (resp. $n$) hence the conjugacy class containing $g$ (resp. $h$) is completely determined. Choose $z$ in the conjugacy class of involutions in $S\backslash T$. We calculate $a_{g,h,z}$ using Proposition \ref{p: character} and in both cases obtain $a_{g,h,z}=12$. In particular we conclude that we can choose $g,h$ so that $gh$ is an involution and $(S,g,h)$ is a $(2,m,n)$-group. Assume that $g$ and $h$ are chosen in this way from here on.

Now suppose that $g'$ (resp. $h'$) are elements such that $(S,g',h')$ is also a $(2,m,n)$-group. We know that $g'h'$ lies in the same conjugacy class as $gh$; thus, by conjugating appropriately, we may assume that $g'h'=gh$. Since $a_{g,h,z}=12$ we know that there are 12 possible pairs $(g',h')$ as given. Using \cite{atlas}, we observe that $|C_S(gh)|=12$ and, moreover, $C_S(gh)\cap C_S(g)=\{1\}$. Thus there exists $x\in C_S(gh)$ such that $g'=g^x, h'=h^x$ and we are done.
\end{proof}

Note that the condition on elements $g', h'$ given in the proposition is similar to the notion of {\it rigidity} which has been studied in other contexts (see \cite{marion}).

We finish with the promised infinite family.

\begin{prop}\label{p: infinite family}
 Let $x=\frac{2^a+9}{29}$ where $a$ is a positive integer satisfying $a\equiv 24 \mod 28$. Then $x$ is an integer and we set $G=SL_2(8)\times D_{2x}$. Then $G$ is a $(2, 7x, 18)$-group with associated Euler characteristic $-2^{a+1}$.
\end{prop}
\begin{proof}
It is a trivial matter to check that $x$ is an integer if and only if $a\equiv 24\mod 28$. Let $g_1\in SL_2(8)$ have order $7$ and $h_1\in SL_2(8)$ have order $9$. Checking \cite{atlas} and applying Proposition \ref{p: character} confirms that we can choose $g_1$ and $h_1$ so that they generate $SL_2(8)$ and so that $g_1h_1$ has order $2$. Now choose $g_1\in D_{2x}$ of order $x$ and $g_2$ of order $2$. Clearly $\langle (g_1, h_1), (g_2, h_2)\rangle =G$. The order of $(g_2, h_2)$ is $18$; since $7$ never divides $\frac{2^a+9}{29}$ for $a$ a positive integer, it follows that the order of $(g_1, h_1)$ is $7x$, and the order of $(g_1, h_1)\cdot(g_2, h_2)$ is easily observed to be $2$.

We calculate the Euler characteristic using (\ref{e: sunny}):
$$\chi = 7\cdot 8\cdot 9\cdot 2x\left(\frac1{7x} + \frac1{18}-\frac12\right) =-58x+18.$$
Now $\chi=-2^{a+1}$ if and only if $x=\frac{2^a+9}{29}$, and we are done.
\end{proof}

\bibliographystyle{amsalpha}
\bibliography{paper}

\end{document}